\documentclass[a4paper,10pt,reqno]{amsart}
\usepackage{amsfonts,amsmath,amssymb,amsthm}
\usepackage{graphicx}
\usepackage{fixmath}
\usepackage{hyperref}
\usepackage{tikz}
\usepackage[font=small]{caption}
\usepackage[font=small, subrefformat=parens]{subcaption}
\usepackage{mathrsfs}
\usepackage{mathtools}
\usepackage{stmaryrd}

\newtheorem{theo}{Theorem}[section]

\newtheorem{lemma}[theo]{Lemma}

\begin{document}

\pagestyle{plain}

\title{The Steiner Wiener index of trees with a given segment sequence}

\author{Jie Zhang}
\address{Jie Zhang\\
School of Insurance\\
Shanghai Lixin University of Accounting and Finance\\
995 Shangchuan Road, Shanghai 201209,  P.R. China
}
\email{zhangjie.sjtu@163.com}

\author{Hua Wang}
\address{Hua Wang\\
Department of Mathematical Sciences \\
Georgia Southern University \\
Statesboro, GA 30460, USA
}
\email{hwang@georgiasouthern.edu}

\author{Xiao-Dong Zhang}
\address{Xiao-Dong Zhang\\
Department of Mathematics, MOE-LSC and SHL-MAC\\
Shanghai Jiao Tong University\\
800 Dongchuan road, Shanghai, 200240, P. R. China
}
\email{xiaodong@sjtu.edu.cn}
\thanks{
Supported by the National Natural Science Foundation of China
(No.11701372; No.11526140); Shanghai Natural Science Foundation (No.16ZR1422400), Shanghai YangFan Program (No.16YF1415900),
Excellent Teacher's Program of Shanghai Municipal Education Commission (No. ZZshjr15027).}

\date{}
\maketitle

\begin{abstract}
The Steiner distance of vertices in a set $S$ is the minimum size of a connected subgraph that contain these vertices. The sum of the Steiner distances over all sets $S$ of cardinality $k$ is called the Steiner $k$-Wiener index and studied as the natural generalization of the famous Wiener index in chemical graph theory. In this paper we study the extremal structures, among trees with a given segment sequence, that maximize or minimize the Steiner $k$-Wiener index. The same extremal problems are also considered for trees with a given number of segments.
\end{abstract}

\maketitle

\section{Introduction}
\label{sec:intro}

With $d(u,v)$ denoting the distance between two vertices $u$ and $v$ in a graph $G$, the well known Wiener index is defined as
\begin{equation}\label{eq:defw}
W(G) = \sum_{\{u,v\} \subset V(G)} d(u,v)
\end{equation}
where the sum is over all possible unordered pairs of vertices in $V(G)$. Introduced in 1947  \cite{wiener1947, wiener1947'}, the Wiener index has become one of the most studied graph invariants in chemical graph theory. Because of the applications related to acyclic molecular structures in chemistry and biochemistry, the Wiener index of trees has been extensively studied. See, for instance, \cite{and2016, dobrynin, laszlo, lin, lin2015, liu2010, liu2008, liu2016, Mukwembi, nina, shi1993, wang, zhang2010, zhang} and the references therein.

A {\it segment} of a tree $T$ is a subpath $P$ of $T$ such that the all internal vertices of $P$ are of degree 2 in $T$ and each end of $P$ is either a leaf or a branch vertex (a vertex of degree at least 3). The {\it segment sequence} $(l_1, l_2, \cdots, l_{m})$ of a tree is simply the sequence of the segment lengths in non-increasing order.
Among various classes of trees that have been studied, the trees with a given segment sequence are of particular interest to us. Along this line, it has been shown in \cite{lin2015} that the {\it starlike tree} (obtained from identifying the ends of all segments) minimizes the Wiener index among trees of a given segment sequence. A tree whose removal of pendant segments (segments with one end being a leaf) results in a path is called a {\it quasi-caterpillar}, which was shown in \cite{and2016} to maximize the Wiener index among trees of a given segment sequence. See Figure~\ref{fig:starcat} for an illustration of these extremal structures.

\begin{figure}[htbp]
 \centering
 \begin{subfigure}{\linewidth}
\centering
\begin{tikzpicture}

        \node[fill=black,circle,inner sep=1pt] (t4) at (0,0) {};
        \node[fill=black,circle,inner sep=1pt] (t5) at (1.76,0.64) {};
        \node[fill=black,circle,inner sep=1pt] (t6) at (2.76,0.64) {};
        \node[fill=black,circle,inner sep=1pt] (t1) at (2,0) {};
        \node[fill=black,circle,inner sep=1pt] (t2) at (1.76,-.64) {};
		
\node[fill=black,circle,inner sep=1pt] (a1) at (1.13,-0.98) {};

\node[fill=black,circle,inner sep=1pt] (p1) at (1, 0){};
\node[fill=black,circle,inner sep=1pt] (p2) at (40:1cm){};
\node[fill=black,circle,inner sep=1pt] (p3) at (2*40:1cm){};
\node[fill=black,circle,inner sep=1pt] (p4) at (3*40:1cm){};
\node[fill=black,circle,inner sep=1pt] (p5) at (4*40:1cm){};
\node[fill=black,circle,inner sep=1pt] (p6) at (5*40:1cm){};
\node[fill=black,circle,inner sep=1pt] (p7) at (6*40:1cm){};
\node[fill=black,circle,inner sep=1pt] (p8) at (7*40:1cm){};
\node[fill=black,circle,inner sep=1pt] (p9) at (8*40:1cm){};

       \draw (p2)--(t5) (t5)--(t6);
       \draw (p1)--(t1) (p9)--(t2);
       \draw (t4)--(p1) (t4)--(p2) (t4)--(p3) (t4)--(p4) (t4)--(p5);
       \draw (a1)--(p8) (t4)--(p6) (t4)--(p7) (t4)--(p8) (t4)--(p9) ;
\end{tikzpicture}
\end{subfigure}

\vspace{3 em}

\begin{subfigure}{\linewidth}
 \centering
\begin{tikzpicture}
        \node[fill=black,circle,inner sep=1pt] (t1) at (-4,0) {};
        \node[fill=black,circle,inner sep=1pt] (t2) at (-3,0) {};
		\node[fill=black,circle,inner sep=1pt] (t3) at (-2,0) {};
        \node[fill=black,circle,inner sep=1pt] (t4) at (-1,0) {};
        \node[fill=black,circle,inner sep=1pt] (t5) at (0,0) {};
        \node[fill=black,circle,inner sep=1pt] (t6) at (1,0) {};
        \node[fill=black,circle,inner sep=1pt] (t7) at (2,0) {};
        \node[fill=black,circle,inner sep=1pt] (a1) at (-1.7,-0.6) {};
        \node[fill=black,circle,inner sep=1pt] (a2) at (-1.7,-1.5) {};
		\node[fill=black,circle,inner sep=1pt] (a3) at (-1,-0.7) {};
        \node[fill=black,circle,inner sep=1pt] (a4) at (-1,-1.4) {};
        \node[fill=black,circle,inner sep=1pt] (a5) at (-1,-2) {};
        \node[fill=black,circle,inner sep=1pt] (a6) at (-.4,-0.7) {};
        \node[fill=black,circle,inner sep=1pt] (a7) at (.4,-0.7) {};
        \node[fill=black,circle,inner sep=1pt] (a8) at (-3,-0.7) {};
        \draw (t1)--(t7);
        \draw (t2)--(a8) (t4)--(a1) (t4)--(a5);
        \draw (t5)--(a6) (t5)--(a7) (a1)--(a2);
\end{tikzpicture}
\end{subfigure}
\caption{The starlike tree (top) and quasi-caterpillar (bottom) with segment sequence $(3,2,2,2,1,1,1,1,1)$.}\label{fig:starcat}
\end{figure}
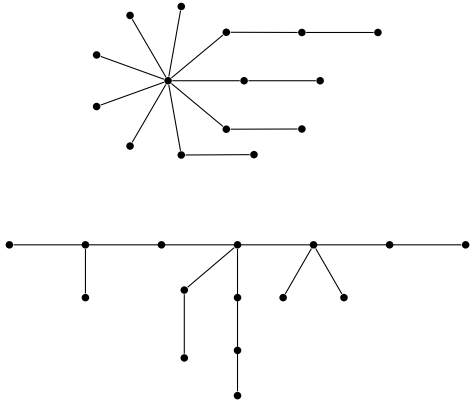

A natural generalization of the distance between two vertices is the {\it Steiner distance} $d(S)$ of the vertices of $S$, defined as the minimum size of a connected subgraph whose vertex set contains $S$. The sum of the Steiner distances (or equivalently, the average Steiner distance) was studied earlier in \cite{Dankelmann96, Dankelmann97} and was recently proposed independently as the generalization of the Wiener index \cite{li2016}. This {\it Steiner $k$-Wiener index} $SW_k(G)$ of
$G$ is defined as
$$SW_k(G)=\sum_{S\subseteq V(G), |S|=k} d(S)  $$
by replacing $d(u,v)$ in \eqref{eq:defw} by $d(S)$ for $|S|=k$. It is easy to see that $SW_2(G)=W(G)$. It was also shown in \cite{li2016} that
$$
SW_3(T)=\frac{n-2}{2}W(T)
$$
in trees.

As a generalization of the Wiener index the Steiner Wiener index has received much attention in the past few years \cite{gutman15, gutman, li2017, lu, mao, mao17, mao2017, wang2018}. We will consider the extremal problems with respect to the Steiner Wiener index in trees with a given segment sequence and show that the starlike tree and the quasi-caterpillar are still extremal. The proof of these statements are, as expected, more complicated than those of their counter parts with respect to the Wiener index. For convenience we will
use $\mathcal{T}_{\ell}$ to denote the set of trees of order $n$ with the segment sequence $\ell = (l_1, l_2, \cdots, l_{m})$.

\begin{theo}\label{theo:theo}
Among trees in $\mathcal{T}_{\ell}$, the starlike tree minimizes the Steiner $k$-Wiener index for any $k$.
\end{theo}

\begin{theo}\label{theo:theo'}
For any $k$ and $\ell$, among trees in $\mathcal{T}_{\ell}$, the Steiner $k$-Wiener index is maximized by a quasi-caterpillar.
\end{theo}

We present the proofs to Theorems~\ref{theo:theo} and \ref{theo:theo'} in Sections~\ref{sec:min} and \ref{sec:max}, respectively. In addition to showing that the tree (in $\mathcal{T}_{\ell}$) with maximum Steiner Wiener index must be a quasi-caterpillar, we further examine its structureal characteristics in Section~\ref{sec:quasi}. We also consider the extremal problems among trees with a given number of segments in Section~\ref{sec:num} and characterize the extremal trees. In Section~\ref{sec:con} we summarize our findings as well as point out some directions for potential future work.

\section{Proof of Theorem~\ref{theo:theo}}
\label{sec:min}

For a tree $T$ in $\mathcal{T}_{\ell}$ that minimizes the Steiner $k$-Wiener index, we call it the \emph{optimal tree}. To prove Theorem~\ref{theo:theo} we only need to show that $T$ has only one branch vertex.

Suppose, for contradiction, that $P(u_1, u_2):=u_1v_1v_2\cdots v_{s-1}u_2$ is a segment with both $u_1$ and $u_2$ being branch vertices.
Let the neighbors of $u_1$ be
$v_1$, $u_{1,1}$, $u_{1,2}$, $\cdots$, $u_{1,t-1}$ and denote by $T_{u_2}$ the component in $T-E(P(u_1,u_2))$ containing $u_2$.
Further let $T_{1,i}$ denote the component containing $u_{1,i}$ after removing the edge between $u_1$ and $u_{1,i}$, for $1\leq i\leq t-1$. For technical reasons we let $Q=T_{1,1}\cup \cdots\cup T_{1,t-1}$ be the union of these components and
$$T'=T-\{u_1u_{1,1} \cdots, u_1u_{1,t-1}\}+\{u_2u_{1,1} \cdots, u_2u_{1,t-1}\}.$$
See Figure~\ref{fig:1} for an illustration with $u_1=v_0$ and $u_2=v_s$.

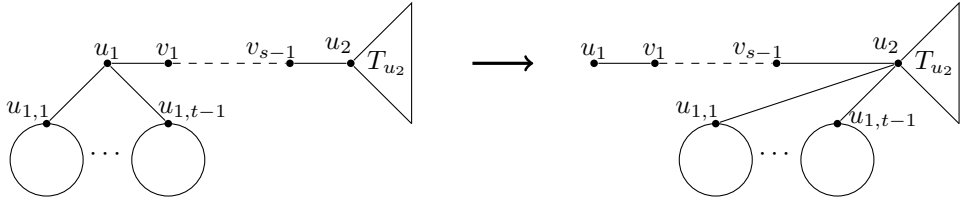
\begin{figure}[htbp]
\centering
    \begin{tikzpicture}[scale=.8]
        \node[fill=black,circle,inner sep=1pt] (t1) at (-1,0) {};
        \node[fill=black,circle,inner sep=1pt] (t2) at (0,0) {};
		\node[fill=black,circle,inner sep=1pt] (t3) at (2,0) {};
        \node[fill=black,circle,inner sep=1pt] (t4) at (3,0) {};
        \node[fill=black,circle,inner sep=1pt] (t5) at (-2,-1) {};
		\node[fill=black,circle,inner sep=1pt] (t6) at (0,-1) {};

        \draw (t1)--(t2) (t3)--(t4);
        \draw  (t1)--(t5);
        \draw  (t1)--(t6);
        \draw [dashed](t2)--(t3);

        \draw  (-2,-1.6) circle[ radius=0.6];
        \draw  (0,-1.6) circle[ radius=0.6];
       \draw  (t4)--(4,1)--(4,-1);
         \draw  (t4)--(4,-1);

        \node at (-1,.2) {$u_{1}$};
		\node at (0,.2) {$v_{1}$};
        \node at (1.7,.2) {$v_{s-1}$};
         \node at (2.8,.3) {$u_{2}$};

        \node at (-2.3,-0.8) {$u_{1,1}$};
        \node at (.4,-0.8) {$u_{1,t-1}$};
        \node at (3.6,0) {$T_{u_2}$};
         \node at (-1,-1.5) {$\cdots$};

         \draw[->,very thick] (5,0) -- (6,0);

        \node[fill=black,circle,inner sep=1pt] (u1) at (7,0) {};
        \node[fill=black,circle,inner sep=1pt] (u2) at (8,0) {};
		\node[fill=black,circle,inner sep=1pt] (u3) at (10,0) {};
        \node[fill=black,circle,inner sep=1pt] (u4) at (12,0) {};
        \node[fill=black,circle,inner sep=1pt] (u5) at (9,-1) {};
		\node[fill=black,circle,inner sep=1pt] (u6) at (11,-1) {};

        \draw (u1)--(u2) (u3)--(u4);
        \draw  (u4)--(u5);
        \draw  (u4)--(u6);
         \draw [dashed](u2)--(u3);

        \draw  (9,-1.6) circle[ radius=0.6];
        \draw  (11,-1.6) circle[ radius=0.6];
       \draw  (u4)--(13,1)--(13,-1);
        \draw  (u4)--(13,-1);

        \node at (7,.2) {$u_{1}$};
		\node at (8,.2) {$v_{1}$};
        \node at (9.7,.2) {$v_{s-1}$};
         \node at (11.8,.3) {$u_{2}$};

        \node at (8.7,-0.8) {$u_{1,1}$};
        \node at (11.75,-0.95) {$u_{1,t-1}$};
        \node at (12.6,0) {$T_{u_2}$};
         \node at (10,-1.5) {$\cdots$};
        \end{tikzpicture}
\caption{The trees $T$ and $T'$.}\label{fig:1}
\end{figure}

For different choices of $S \subset V(T)$, we now consider the change in $d(S)$ from $T$ to $T'$. We use $d_T(S)$ and $d_{T'}(S)$ to distinguish between the underlying tree structure. Sometimes we also use $|G|$ for $|V(G)|$ for a graph $G$.

\begin{itemize}
\item If $S\cap V(T_{u_2}-\{u_2\})=\emptyset$, by noting that
$T-(T_{u_2}-\{u_2\})$ and $T'-(T_{u_2}-\{u_2\})$ are isomorphic to each other, we have
$$\sum _{S\cap V(T_{u_2}-\{u_2\})=\emptyset}d_T(S)=\sum _{S\cap V(T_{u_2}-\{u_2\})=\emptyset}d_{T'}(S);$$

\item If $S\cap V(Q) = \emptyset$, it is easy to see that $d_T(S) = d_{T'}(S)$;

\item If $S\cap V(T_{u_2}-\{u_2\})\neq \emptyset$ and $S\cap V(Q) \neq \emptyset$:
\begin{itemize}
\item If $S \cap \{v_0, v_1, \cdots, v_s\} \neq \emptyset$, suppose $|S\cap V(Q)|=a$ and $|S\cap V(T_{u_2}-\{u_2\})|=b$ for some $a\geq 1$ and $1\leq b\leq k-a-1$. Further let $i_0$ be the smallest integer such that $v_{i_0} \in \{v_0, v_1, \cdots, v_s\}$ is in $S$, then
$$d_{T'}(S)=d_{T}(S)-i_0 \leq d_T(S) ;$$
\item If $S \cap \{v_0, v_1, \cdots, v_s\} =\emptyset$, then
$$d_{T'}(S)=d_{T}(S)-s < d_T(S) .$$
\end{itemize}
\end{itemize}

Summarizing all cases, we now have
$$ SW_k(T') - SW_k(T) < 0, $$
a contradiction.

\section{Maximizing $SW_k(T)$ in $\mathcal{T}_{\ell}$}
\label{sec:max}

In this section we will consider the problem of maximizing $SW_k(T)$ in $\mathcal{T}_{\ell}$. First we present a ``switching'' operation that increases the value of the Steiner $k$-Wiener index.

Let $P(w_0, w_s):= w_0w_1\cdots, w_s$ be a segment in $T$ with both $w_0$ and $w_s$ being branch vertices, and let the neighborhoods $N_T(w_0)=\{w_1, w_{0,1}, \cdots\}$ and $N_T(w_s)=\{w_{s-1}, w_{s,1}, \cdots\}$.
Similar to before we use $T_{0,1}$ $(T_{s,1})$ to denote the component containing $w_{0,1}$ $(w_{s,1})$ after removing the edge between $w_0w_{0,1}$
($w_sw_{s,1}$) from $T$, and $T_{w_0}$ $(T_{w_s})$ to denote the component in $T-E(P(w_0,w_s))$ that contains $w_0$ ($w_s$).
For simplicity we introduce the labels $X=V(T_{w_0}-T_{0,1})$, $Y=V(T_{w_s}-T_{s,1})$, $A=V(T_{0,1})$, and $B=V(T_{s,1})$ (Figure~\ref{fig:2}).

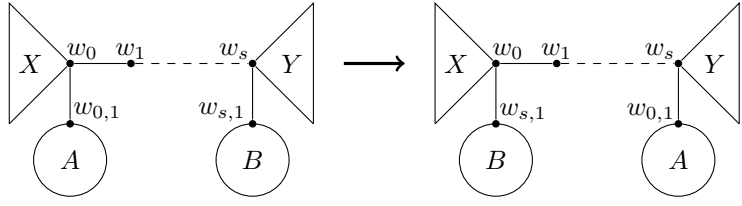
\begin{figure}[htbp]
\centering
    \begin{tikzpicture}[scale=.8]
        \node[fill=black,circle,inner sep=1pt] (t1) at (-1,0) {};
        \node[fill=black,circle,inner sep=1pt] (t2) at (0,0) {};
		\node[fill=black,circle,inner sep=1pt] (t3) at (2,0) {};
        \node[fill=black,circle,inner sep=1pt] (t5) at (-1,-1) {};
		\node[fill=black,circle,inner sep=1pt] (t6) at (2,-1) {};

        \draw (t1)--(t2);
        \draw [dashed](t2)--(t3);
        \draw  (t1)--(t5);
        \draw  (t3)--(t6);

        \draw  (-1,-1.6) circle[ radius=0.6];
        \draw  (2,-1.6) circle[ radius=0.6];
       \draw  (t1)--(-2,1)--(-2,-1);
       \draw  (t1)--(-2,-1);
       \draw  (t3)--(3,1)--(3,-1);
       \draw  (t3)--(3,-1);

        \node at (-0.8,.2) {$w_{0}$};
		\node at (0,.2) {$w_{1}$};
        \node at (1.7,.2) {$w_{s}$};

        \node at (-1.65,0) {$X$};
        \node at (2.65,0) {$Y$};
        \node at (-0.55,-0.8) {$w_{0,1}$};
        \node at (1.5,-0.8) {$w_{s,1}$};
        \node at (-1,-1.6) {$A$};
        \node at (2,-1.6) {$B$};

         \draw[->,very thick] (3.5,0) -- (4.5,0);

       \node[fill=black,circle,inner sep=1pt] (t1) at (6,0) {};
        \node[fill=black,circle,inner sep=1pt] (t2) at (7,0) {};
		\node[fill=black,circle,inner sep=1pt] (t3) at (9,0) {};
        \node[fill=black,circle,inner sep=1pt] (t5) at (6,-1) {};
		\node[fill=black,circle,inner sep=1pt] (t6) at (9,-1) {};

        \draw (t1)--(t2);
        \draw [dashed](t2)--(t3);
        \draw  (t1)--(t5);
        \draw  (t3)--(t6);

        \draw  (6,-1.6) circle[ radius=0.6];
        \draw  (9,-1.6) circle[ radius=0.6];
       \draw  (t1)--(5,1)--(5,-1);
        \draw  (t1)--(5,-1);
       \draw  (t3)--(10,1)--(10,-1);
       \draw  (t3)--(10,-1);

        \node at (6.2,.2) {$w_{0}$};
		\node at (7,.2) {$w_{1}$};
        \node at (8.7,.2) {$w_{s}$};

        \node at (5.35,0) {$X$};
        \node at (9.6,0) {$Y$};
        \node at (6.45,-0.83) {$w_{s,1}$};
        \node at (8.55,-0.83) {$w_{0,1}$};
        \node at (6,-1.6) {$B$};
        \node at (9,-1.6) {$A$};

        \end{tikzpicture}
\caption{The tree $T$ (on the left) and $T'$ (on the right) after the ``switching'' operation.}\label{fig:2}
\end{figure}

\begin{lemma}\label{lemma3.1}
Let $T'$ be obtained from $T$ by ``switching'' $A$ and $B$ (Figure~\ref{fig:2}). If $|X|> |Y|$ and $|A|> |B|$, then $SW_k(T')> SW_k(T)$.
\end{lemma}

\begin{proof}
Similar to before, we examine the change in $d(S)$, from $T$ to $T'$, depending on different choices of $S$. It is easy to see that $d_{T'}(S) = d_T(S)$ when $S$ does not contain any vertex from $A\cup B$ or when $S$ contains vertices from both $A$ and $B$. In what follows we assume that $S$ contains vertices from exactly one of $A$ and $B$. Similarly we only consider those choices of $S$ that contain vertices from at most one of $X$ and $Y$ (for otherwise $d(S)$ does not change from $T$ to $T'$).

\begin{itemize}
\item If $\{w_1,\cdots, w_{s-1}\}\not\in S$:

\begin{itemize}
\item $|S\cap A|=a$ for some $1\leq a \leq k-1$:
\begin{itemize}
\item If $|S\cap X|=k-a\geq 1,$ then $d_{T'}(S)=d_{T}(S)+s$;
\item If $|S\cap Y|=k-a\geq 1,$ then $d_{T'}(S)=d_{T}(S)-s$.
\end{itemize}
\item $|S\cap B|=a$ for some $1 \leq a \leq k-1$:
\begin{itemize}
\item If $|S\cap X|=k-a\geq 1,$ then $d_{T'}(S)=d_{T}(S)-s$;
\item If $|S\cap Y|=k-a\geq 1,$ then $d_{T'}(S)=d_{T}(S)+s$.
\end{itemize}
\end{itemize}

Summing over all cases so far, the total change in the sum of the values of $d(S)$ is
\begin{align*}
\Delta_1  & := s\cdot \left[ \sum_{a=1}^{k-1}\left(\binom{|A|}{a}-\binom{|B|}{a}\right)\left(\binom{|X|}{k-a}-\binom{|Y|}{k-a}\right) \right] > 0.
\end{align*}

\item If $|\{w_1,\cdots, w_{s-1}\}\cap S|=b$ for some $b \geq 1$. Let $i\geq 1$ be the smallest index such that $w_i \in S$, and $j\leq s-1$ be the largest index such that $w_j \in S$.
\begin{itemize}
\item If $S \cap (X \cup Y) = \emptyset$:
\begin{itemize}
\item $|S\cap A|=k-b\geq 1$, then $d_{T'}(S)=d_{T}(S)+s-i-j$.
\item $|S\cap B|=k-b\geq 1$, then $d_{T'}(S)=d_{T}(S)+i+j-s$.
\end{itemize}
Hence, for the above two subcases, the total change in the sum of the values of $d(S)$ is
\begin{align*}
\Delta_2^{i,j} & := (s-i-j)\cdot\sum_{b=1}^{k-1}\left(\binom{|A|}{k-b}-\binom{|B|}{k-b}\right)\cdot\binom{s-1}{b}.
\end{align*}
\item If  $S \cap (X \cup Y) \neq \emptyset$:
\begin{itemize}
\item $|S\cap A|=a\geq 1$, $|S\cap X|=k-a-b\geq 1,$ then $d_{T'}(S)=d_{T}(S)+s-j$.
\item $|S\cap A|=a\geq 1$, $|S\cap Y|=k-a-b\geq 1,$ then $d_{T'}(S)=d_{T}(S)-i$.
\item $|S\cap B|=a\geq 1$, $|S\cap X|=k-a-b\geq 1,$ then $d_{T'}(S)=d_{T}(S)+j-s$.
\item $|S\cap B|=a\geq 1$, $|S\cap Y|=k-a-b\geq 1,$ then $d_{T'}(S)=d_{T}(S)+i$.
\end{itemize}
Hence, for the above four subcases, the total change in the sum of the values of $d(S)$ is
\begin{align*}
\Delta_3 & := \sum_{a=1}^{k-2}\sum_{b=1}^{k-1-a}\left(s-j\right)\binom{|X|}{k-a-b}\binom{s-1}{b}\left(\binom{|A|}{a}-\binom{|B|}{a}\right)\\
 & \quad\quad -\sum_{a=1}^{k-2}\sum_{b=1}^{k-1-a}i\binom{|Y|}{k-a-b}\binom{s-1}{b}\left(\binom{|A|}{a}-\binom{|B|}{a}\right)\\
 & = \sum_{a=1}^{k-2}\sum_{b=1}^{k-1-a}\left(s-i-j\right)\binom{|X|}{k-a-b}\binom{s-1}{b}\left(\binom{|A|}{a}-\binom{|B|}{a}\right)\\
 & \quad\quad +\sum_{a=1}^{k-2}\sum_{b=1}^{k-1-a}i\binom{s-1}{b}\left(\binom{|X|}{k-a-b}-\binom{|Y|}{k-a-b}\right)\left(\binom{|A|}{a}-\binom{|B|}{a}\right) \\
& = \Delta_{3_a}^{i,j} + \Delta_{3_b}^i .
\end{align*}
\end{itemize}
\end{itemize}

It is easy to see that $\Delta_{3_b}^i > 0$. Then
\begin{align*}
SW_k(T')-SW_k(T) & = \Delta_1 + \sum_{1\leq i \leq j \leq s-1} \Delta_2^{i,j} + \sum_{1\leq i \leq j \leq s-1} \Delta_{3_a}^{i,j} + \sum_{1\leq i \leq s-1} \Delta_{3_b}^{i} \\
& > \sum_{1\leq i \leq j \leq s-1} \Delta_2^{i,j} + \sum_{1\leq i \leq j \leq s-1} \Delta_{3_a}^{i,j} \\
& = 0
\end{align*}
where the last identity follows from the fact that
$$ \Delta_2^{i,j} + \Delta_2^{s-j, s-i} =  \Delta_{3_a}^{i,j} + \Delta_{3_a}^{s-j, s-i} = 0 $$
for any choices of $i$ and $j$.
\end{proof}

We may now apply this ``switching'' operation to prove the main result of this section, that the Steiner $k$-Wiener index is maximized by a quasi-caterpillar.

\begin{proof}[Proof of Theorem~\ref{theo:theo'}]
Again we call the extremal tree optimal and let $T$ be such a tree. Consider a path $P$ in $T$, with the greatest number of segments on it, with end vertices $v_0$ and $v_{k}$ (both of which have to be leaves). We now denote, in the order of their distances from $v_0$, the branch vertices on $P$ by $v_1,v_2,\ldots,v_{k-1}$. Furthermore, for each $i$ ($1\leq i \leq k-1$) we let the neighbors of $v_i$ that do not lie on $P$ be $v_{i,1}, v_{i,2}, v_{i,3}, \ldots $, and let $T_{i,j}$ denote the component containing $v_{i,j}$ after removing the edge between $v_iv_{i,j}$. For convenience we use $T_{\geq t}$ $(T_{\leq t})$ for the subtree
induced by $\cup_{i\geq t}V(T_i)$ $(\cup_{i\leq t}V(T_i))$. Within each $T_{i,j}$ we let $u_{i,j}$ denote the branch vertex closest to $v_{i,j}$ and let $Q_{i,j}$ denote the component containing $u_{i,j}$ after removing the edges on the path from $v_i$ to $u_{i,j}$. See Figure~\ref{fig:ex9} for an illustration.

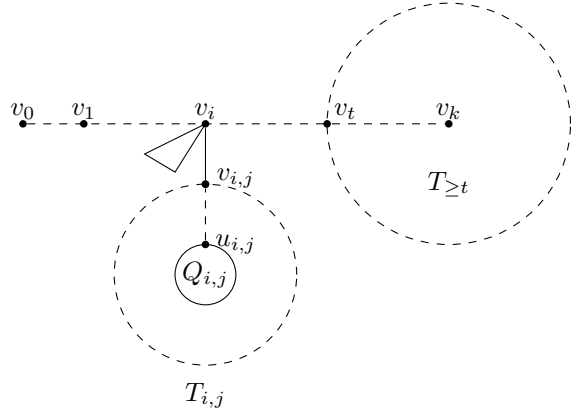
\begin{figure}[htbp]
\centering
    \begin{tikzpicture}[scale=.8]
        \node[fill=black,circle,inner sep=1pt] (t1) at (-1,0) {};
        \node[fill=black,circle,inner sep=1pt] (t2) at (0,0) {};
		\node[fill=black,circle,inner sep=1pt] (t3) at (2,0) {};
        \node[fill=black,circle,inner sep=1pt] (t4) at (4,0) {};
        \node[fill=black,circle,inner sep=1pt] (t9) at (6,0) {};
        \node[fill=black,circle,inner sep=1pt] (t5) at (2,-1) {};
        \node[fill=black,circle,inner sep=1pt] (t6) at (2,-2) {};
         \node[fill=black,circle,inner sep=1pt] (t8) at (2,-1) {};
        \draw [dashed](t1)--(t9)(t6)--(t5);
        \draw  (t3)--(t5) ;

        \draw [dashed](6,0)  circle [radius=2];
        \draw (2,-2.5)  circle [radius=.5];
        \draw [dashed] (2,-2.5)  circle [radius=1.5];
        \draw (2,0)--(1,-.5)--(1.5,-0.8)--cycle;

        \node at (4.3,.2) {$v_{t}$};
        \node at (2,.2) {$v_{i}$};
        \node at (0,.2) {$v_{1}$};
        \node at (-1,.2) {$v_{0}$};
        \node at (6,.2) {$v_{k}$};

                \node at (6,-1) {$T_{\geq t}$};
                \node at (2.5,-0.9) {$v_{i,j}$};
                \node at (2.5,-2) {$u_{i,j}$};
                 \node at (2,-2.5) {$Q_{i,j}$};
                 \node at (2,-4.5) {$T_{i,j}$};

        \end{tikzpicture}
\caption{Labeling of the path $P(v_0, v_k)$ and related objects.}\label{fig:ex9}
\end{figure}

Suppose, for contradiction, that $T$ is not a quasi-caterpillar. Then there is at least one $Q_{i,j}$ that is not a single vertex.
Let $Q = Q_{i_0,j_0}$ be such a component and suppose, without loss of generality, that
\begin{equation}\label{eq:branchsizes}
 |T_{\leq i_0-1}| \geq |T_{\geq i_0+1}|.
\end{equation}
Note that we can further assume that $i_0$ is the largest index among those of all such components. That is, $|Q_{i,j}|=1$ for all $i > i_0$.

By our choice of the path $P = P(v_0,v_k)$ (as the one with the most number of segments) it is easy to see that $i_0 \neq k-1$. Thus $v_{i_0+1}$ is still a branch vertex (Figure~\ref{fig:ex1}).

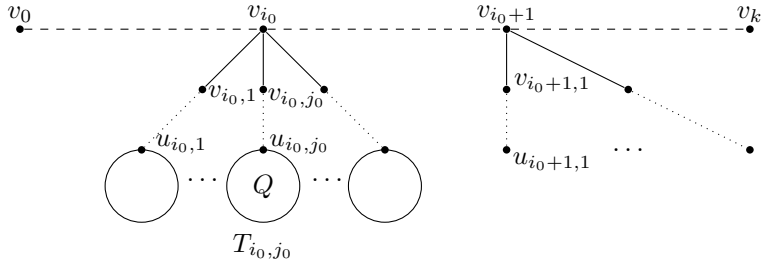
\begin{figure}[htbp]
\centering
    \begin{tikzpicture}[scale=.8]
        \node[fill=black,circle,inner sep=1pt] (t1) at (-4,0) {};
        \node[fill=black,circle,inner sep=1pt] (t2) at (0,0) {};
		\node[fill=black,circle,inner sep=1pt] (t3) at (4,0) {};
        \node[fill=black,circle,inner sep=1pt] (t4) at (8,0) {};
        \node[fill=black,circle,inner sep=1pt] (t5) at (-1,-1) {};
		\node[fill=black,circle,inner sep=1pt] (t6) at (0,-1) {};
		\node[fill=black,circle,inner sep=1pt] (t7) at (1,-1) {};
		\node[fill=black,circle,inner sep=1pt] (t8) at (4,-1) {};
		\node[fill=black,circle,inner sep=1pt] (t9) at (6,-1) {};

        \node[fill=black,circle,inner sep=1pt] (u1) at (-2,-2) {};
         \node[fill=black,circle,inner sep=1pt] (u2) at (0,-2) {};
          \node[fill=black,circle,inner sep=1pt] (u3) at (2,-2) {};
          \node[fill=black,circle,inner sep=1pt] (u4) at (4,-2) {};
           \node[fill=black,circle,inner sep=1pt] (u5) at (8,-2) {};

        \draw [dashed](t2)--(t3);
         \draw [dashed] (t1)--(t2);
          \draw [dashed] (t3)--(t4);
        \draw  (t2)--(t5);
        \draw  (t2)--(t6);
        \draw (t2)--(t7);
        \draw (t3)--(t8);
        \draw (t3)--(t9);
        \draw [dotted] (t5)--(u1);
        \draw [dotted] (t6)--(u2);
        \draw [dotted] (t7)--(u3);
        \draw [dotted] (t8)--(u4);
        \draw [dotted] (t9)--(u5);

        \draw  (-2,-2.6) circle[ radius=0.6];
        \draw  (0,-2.6) circle[ radius=0.6];
        \draw  (2,-2.6) circle[ radius=0.6];

        \node at (0,.25) {$v_{i_0}$};
		\node at (4,.25) {$v_{i_0+1}$};
        \node at (8,.25) {$v_{k}$};
        \node at (-4,.25) {$v_{0}$};

        \node at (-0.5,-1.1) {$v_{i_0,1}$};
        \node at (0.55,-1.15) {$v_{i_0,j_0}$};
        \node at (4.75,-0.9) {$v_{i_0+1,1}$};

        \node at (-1.35,-1.9) {$u_{i_0,1}$};
        \node at (0.6,-1.9) {$u_{i_0,j_0}$};
        \node at (4.75,-2.2) {$u_{i_0+1,1}$};
          \node at (1.05,-2.5) {$\ldots$};
         \node at (6,-2) {$\ldots$};
         \node at (-1,-2.5) {$\ldots$};

            \node at (0,-2.6) {$Q$};
             \node at (0,-3.6) {$T_{i_0,j_0}$};

        \end{tikzpicture}
\caption{Labeling of the subtree $T_{i_0,j_0}$ and related objects.}\label{fig:ex1}
\end{figure}

For simplicity we let $X=V(T_{\leq i_0}-T_{i_0,j_0}-\{v_{i_0}\})$ and $Y=V(T_{\geq i_0+1}-T_{i_0+1,1}-\{v_{i_0+1}\})$.

We then relabel the vertices of $P(u_{i_0,j_0}, u_{i_0+1,1})$ as  $u_0,\cdots, u_p,\cdots, u_{p+q},\cdots, u_{p+q+p'}$ such that
$u_0=u_{i_0,j_0}$, $u_p=v_{i_0}$, $u_{p+q}=v_{i_0+1}$, and $u_{p+q+p'}=u_{i_0+1,1}$. See the illustration on the left of Figure~\ref{fig:3} where $p$, $q$, and $p'$ are the lengths of $P(v_{i_0},u_{i_0,j_0})$, $P(v_{i_0}, v_{i_0+1})$, and $P(v_{i_0+1},u_{i_0+1,1})$, respectively.


We will show that ``moving'' $Q$ from $u_0=u_{i_0,j_0}$ to $u_{p+q+p'}=u_{i_0+1,1}$ will increase the value of the Steiner $k$-Wiener index.

First consider the case $p\geq p'$, let $T'$ be obtained from $T$ by switching ``switching'' $X$ and $Y$. Then direct application of Lemma~\ref{lemma3.1} implies $SW_k(T')> SW_k(T)$.

If $p < p'$, we consider the tree $T_1$ obtained from $T$ by
``sliding'' the component $G_Z$ induced by $Z:= X\cup V(P(u_p, u_{p+q})) \cup Y$ from $u_p$ to $u_{p'}$ (Figure~\ref{fig:3}).

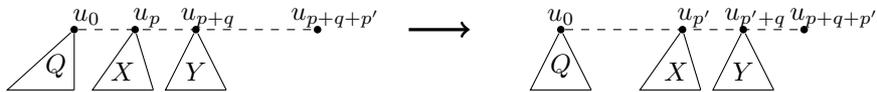
\begin{figure}[htbp]
\centering
    \begin{tikzpicture}[scale=.8]
        \node[fill=black,circle,inner sep=1pt] (t1) at (-2,0) {};
        \node[fill=black,circle,inner sep=1pt] (t2) at (-1,0) {};
		\node[fill=black,circle,inner sep=1pt] (t3) at (0,0) {};
        \node[fill=black,circle,inner sep=1pt] (t4) at (2,0) {};

        \draw [dashed] (t1)--(t4);

      \draw  (t3)--(-0.5,-1)--(0.5,-1);
      \draw  (t1)--(-3.1,-1)--(-2,-1);
       \draw  (t2)--(-1.7,-1)--(-0.7,-1);
       \draw  (t3)--(0.5,-1);
      \draw  (t1)--(-2,-1);
       \draw  (t2)--(-0.7,-1);

        \node at (-1.8,.2) {$u_{0}$};
        \node at (-0.8,.2) {$u_{p}$};
		\node at (0.2,.2) {$u_{p+q}$};
        \node at (2.3,.2) {$u_{p+q+p'}$};

        \node at (-2.3,-.6) {$Q$};
        \node at (-1.2,-0.7) {$X$};
         \node at (0,-0.7) {$Y$};

         \draw[->,very thick] (3.5,0) -- (4.5,0);

   \node[fill=black,circle,inner sep=1pt] (t1) at (6,0) {};
        \node[fill=black,circle,inner sep=1pt] (t2) at (8,0) {};
		\node[fill=black,circle,inner sep=1pt] (t3) at (9,0) {};
        \node[fill=black,circle,inner sep=1pt] (t4) at (10,0) {};

        \draw [dashed] (t1)--(t4);

      \draw  (t3)--(8.5,-1)--(9.5,-1);
       \draw  (t1)--(5.5,-1)--(6.5,-1);
       \draw  (t2)--(7.3,-1)--(8.3,-1);
      \draw  (t3)--(9.5,-1);
       \draw  (t1)--(6.5,-1);
       \draw  (t2)--(8.3,-1);

        \node at (6,.2) {$u_{0}$};
        \node at (8.2,.2) {$u_{p'}$};
		\node at (9.2,.2) {$u_{p'+q}$};
        \node at (10.5,.2) {$u_{p+q+p'}$};

        \node at (6,-.6) {$Q$};
        \node at (7.9,-0.7) {$X$};
         \node at (9,-0.7) {$Y$};
        \end{tikzpicture}
\caption{``Sliding'' the component $G_Z$ from $u_p$ to $u_{p'}$ to generate $T_1$ (on the right).}\label{fig:3}
\end{figure}

As before we consider the change of $d(S)$ from $T$ to $T_1$ depending on the choices of $S$.

It is easy to see that $d(S)$ stays the same if $S \cap Q = \emptyset$ or if $S$ does not contain any vertex in $Z$. Hence in what follows we will assume $S\cap Q \neq \emptyset$ and $S\cap Z \neq \emptyset$.

For simplicity we will only present the proof to the case of $p'>p+q$. The case of $p' \leq p+q$ is similar and we skip the details.

Let $|S\cap V(Q)|=a \geq 1$ and $|S\cap Z|=b \geq 1$, we consider the following cases:
\begin{itemize}
\item If $|S\cap V(P(u_{p+q},u_{p+q+p'}))|=0$, then $d_{T_1}(S)=d_T(S)+p'-p-q$.
\item If $S$ contains some vertices on $P(u_{p+q},u_{p+q+p'})$, let $u_s$ be one with the largest index:
\begin{itemize}
\item If $u_s\in P(u_{p+q},u_{p'})$, $p+q\leq s\leq p'-1$, then $d_{T_1}(S)=d_T(S)+p'-s$.
\item If $u_s \in P(u_{p'},u_{p'+q})$, $p'\leq s \leq p'+q-1$, then $d_{T_1}(S)=d_T(S)+p'+q-s$.
\item If $u_s\in P(u_{p'+q},u_{p+q+p'})$, then $d_{T_1}(S)=d_T(S)$.
\end{itemize}
\end{itemize}
It is easy to see that $d_{T_1}(S) \geq d_T(S)$ in each of the above cases. Thus
$$ SW_k(T_1) - SW_k(T) \geq 0. $$

We now consider $T'$ obtained from $T_1$ by ``switching'' $X$ and $Y$, direct application of Lemma~\ref{lemma3.1} implies $SW_k(T')> SW_k(T_1) \geq SW_k(T)$.

Note that $T'$ is exactly the result of ``moving'' $Q$ from $u_0=u_{i_0,j_0}$ to $u_{p+q+p'}=u_{i_0+1,1}$ in $T$ and the segment sequence is preserved from $T$ to $T'$, we have a contradiction to the
optimality of $T$.
\end{proof}

\section{Further characterization of the extremal quasi-caterpillar}
\label{sec:quasi}

In this section we further discuss the characteristics of the extremal quasi-caterpillars. First let the longest path of a quasi-caterpillar containing all the branch vertices be called the {\it backbone}; all segments that do not lie on the backbone (and thus connect a leaf with a branch vertex) are called \emph{pendant segments}.

\begin{theo}\label{thm:further}
Let $T$ be a quasi-caterpillar that maximizes the Steiner $k$-Wiener index in $\mathcal{T}_{\ell}$, and let the backbone be $P(v_0,v_k)$ between leaves $v_0$ and $v_k$ with branch vertices $v_1, v_2, \ldots, v_{k-1}$ (in the order of their distances from $v_0$). Then $T$
must satisfy the following:
\begin{enumerate}
\item All branch vertices are of degree 3 or 4, the only branch vertices that could possibly have degree 4 are $v_1$ and $v_{k-1}$.
\item The lengths of the segments on the backbone, listed from one end to the other, form a unimodal sequence $r_1, r_2, \ldots ,r_k$, i.e.,
$$r_1 \leq r_2 \leq \cdots \leq r_{j} \geq  \cdots \geq r_k $$
for some $j \in \{1,2,\ldots,k\}$;
\item The lengths of the pendant segments, starting from one end of the backbone towards the other, form a sequence of values $s_1,s_2,\ldots,s_{k'}$ such that
$$ s_1 \geq s_2 \geq \cdots \geq s_{j'} \leq  \cdots \leq s_{k'} $$
for some $j' \in \{1,2,\ldots,k'\}$.
\end{enumerate}
\end{theo}

\begin{proof}

Part of the proof here is very similar to that of the analogous statement (for the regular Wiener index) in \cite{and2016} and we skip some details.

\begin{enumerate}

\item First we show that no branch vertex is of degree $\geq 5$. Otherwise, let $v_i$ be with neighbors $v_{i1}, v_{i2}, v_{i3}, \ldots$ not on $P(v_0,v_k)$. Let $T_{\leq i}$ ($T_{\geq i}$) be defined in the same way as before and let $T_{i1}, T_{i2}, T_{i3}$ be the pendant segments at $v_i$ containing $v_{i1},v_{i2},v_{i3}$ respectively.

Suppose, without loss of generality, that
$$ |T_{\leq  i-1}| \geq |T_{\geq  i+1}|  $$
and hence
$$ |T_{\leq i} - T_{i1} - T_{i2}| > |T_{> i}| . $$

Let $T' \in \mathcal{T}_{\ell}$ be obtained from $T$ by detaching $T_{i1}$ and $T_{i2}$ from $v_i$ and reattaching them to $v_{i+1}$. Lemma~\ref{lemma3.1} implies that $SW_k(T') >SW_k(T)$, a contradiction.

Note that the above argument can be applied to a vertex $v_i$ of degree 4 (moving only one segment instead of two) to obtain a contradiction, unless $v_i = v_1$ or $v_i = v_{k-1}$. Thus the only branch vertices that could possibly have degree 4 are $v_1$ and $v_{k-1}$.

\item Let $r_1, \ldots,r_k$ be the lengths of the segments $P(v_0,v_1), \ldots , P(v_{k-1}, v_k)$ on the backbone, and let $M$ be be the maximum length among them. Suppose, without loss of generality, that not all segments on the backbone are of the same length, and let $j$ be the smallest index such that $r_j = d(v_{j-1}, v_{j})= M > r_{j+1} = d(v_{j}, v_{j+1})$.

Now let $T_{j}$ denote the pendant segment at $v_{j}$ and relabel the vertices on $P(v_{j-1},v_{j+1})$ with $u_0u_1\cdots u_{p'}\cdots u_{p+p'}$ such that
 $u_0=v_{j-1}$, $u_{p'}=v_j$, and $u_{p+p'}=v_{j+1}$, where
$p'=r_j$ and $p=r_{j+1}$.


If $|T_{\leq j-1}| < |T_{\geq j+1}|$, we consider the tree $T' \in \mathcal{T}_{\ell}$, obtained from $T$ by ``Sliding'' $T_j$ from $u_{p'}$ to $u_p$ (Figure~\ref{fig:exslide}).

\begin{figure}[htbp]
\centering
    \begin{tikzpicture}[scale=.8]
        \node[fill=black,circle,inner sep=1pt] (t1) at (-1,0) {};
        \node[fill=black,circle,inner sep=1pt] (t2) at (0,0) {};
        \node[fill=black,circle,inner sep=1pt] (t3) at (1,0) {};
         \node[fill=black,circle,inner sep=1pt] (t4) at (2,0) {};

        \draw [dashed] (t1)--(t4);
        \draw  (t1)--(-2.4,-1)--(-2.4,1);
        \draw  (t3)--(1.5,-1)--(0.5,-1);
        \draw  (t4)--(3.4,-1)--(3.4,1);
         \draw  (t1)--(-2.4,1);
        \draw  (t3)--(0.5,-1);
        \draw  (t4)--(3.4,1);

        \node at (-0.8,.2) {$u_0$};
        \node at (0,.2) {$u_p$};
         \node at (0.8,.2) {$u_{p'}$};
          \node at (1.7,.3) {$u_{p+p'}$};

        \node at (1,-0.7) {$T_{j}$};

        \node at (-1.7,0.02) {$T_{\leq j-1}$};
        \node at (2.8,-0.03) {$T_{\geq j+1}$};

      \draw[->,very thick] (3.8,0) -- (4.5,0);

       \node[fill=black,circle,inner sep=1pt] (t5) at (6,0) {};
        \node[fill=black,circle,inner sep=1pt] (t6) at (7,0) {};
        \node[fill=black,circle,inner sep=1pt] (t7) at (8,0) {};
         \node[fill=black,circle,inner sep=1pt] (t8) at (9,0) {};

        \draw [dashed] (t5)--(t8);
        \draw  (t5)--(4.6,-1)--(4.6,1);
        \draw  (t8)--(10.6,-1)--(10.6,1);
        \draw  (t6)--(6.5,-1)--(7.5,-1);
        \draw  (t5)--(4.6,1);
        \draw  (t8)--(10.6,1);
        \draw  (t6)--(7.5,-1);

        \node at (6.2,.2) {$u_0$};
        \node at (7,.2) {$u_p$};
         \node at (8,.3) {$u_{p'}$};
          \node at (8.9,.35) {$u_{p+p'}$};

        \node at (7,-0.7) {$T_{j}$};

        \node at (5.2,0) {$T_{\leq j-1}$};
        \node at (10,-0.1) {$T_{\geq j+1}$};

        \end{tikzpicture}
\caption{``Sliding'' $T_j$ from $u_{p'}$ to $u_p$.}\label{fig:exslide}
\end{figure}
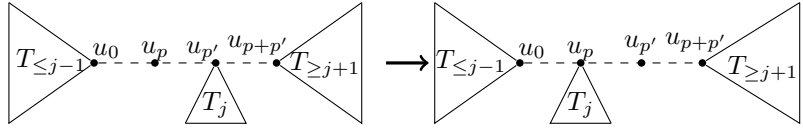

Here we use $V'(T_j)$ to denote $V(T_j)$ without the vertex on the path from $u_0$ to $u_{p+p'}$.
From $T$ to $T'$, $d(S)$ only changes if $S$ contains vertices from one and only one of $T_{\leq j-1}$ and $T_{\geq j+1}$, and if $S \cap V'(T_j) \neq \emptyset$:
\begin{itemize}

\item If $|S\cap V'(T_j)|=a\geq 1$ and $S$ does not contain any other vertices on $P(u_0, u_{p+p'})$:

\begin{itemize}
\item If $|S\cap V(T_{\leq j-1})|=k-a$, then $d_{T'}(S)=d_T(S)-(p'-p).$

\item If $|S\cap V(T_{\geq j+1})|=k-a$, then $d_{T'}(S)=d_T(S)+(p'-p).$
\end{itemize}

The change of sum of values in $d(S)$ is
$$ \sum_{a=1}^{k-1} \left(p'-p\right)\binom{|T_j|-1}{a}\left(\binom{|T_{\geq j+1}|}{k-a}-\binom{|T_{\leq j-1}|}{k-a}\right) > 0 . $$

\item If $|S\cap V'(T_j)|=a\geq 1$, and $S$ contains some other vertices from one and only one of $P(u_0,u_p)$, $P(u_p,u_{p'})$ and $P(u_{p'},u_{p+p'})$,
for convenience, denote $V'(P(u_0, u_p))=V(P(u_0, u_p)-\{u_0\}-\{u_p\}), V'(P(u_{p'}, u_{p+p'}))=V(P(u_{p'}, u_{p+p'})-\{u_{p'}\}-\{u_{p+p'}\})$:

\begin{itemize}
\item If $|S\cap V(T_{\leq j-1})|=b\geq 1$ and $|S \cap V'(P(u_0,u_p))|=k-a-b\geq 1$, then
$d_{T'}(S)=d_T(S)-(p'-p).$

\item If $|S\cap V(T_{\geq j+1})|=b\geq 1$ and $|S \cap V'(P(u_0,u_p))|=k-a-b\geq 1$, then
$d_{T'}(S)=d_T(S).$

\item If $|S\cap V(T_{\leq j-1})|=b\geq 1$ and $|S \cap V(P(u_p,u_{p'}))|=k-a-b\geq 1$,
let  $u_{p+x_1}$ $(u_{p+x_2})$ be the vertices with the smallest (largest) index on $P(u_p,u_{p'})$ that is in $S$,
for some $0\leq x_1\leq x_2\leq p'-p$. Then
$d_{T'}(S)=d_T(S)+x_2-(p'-p).$

\item If $|S\cap V(T_{\geq j+1})|=b\geq 1$ and $|S\cap V(P(u_p,u_{p'}))|=k-a-b\geq 1$, then for similarly defined $x_1$ and $x_2$ we have
$d_{T'}(S)=d_T(S)+x_1.$

\item If $|S\cap V(T_{\leq j-1})|=b\geq 1$ and $|S \cap V'(P(u_{p'},u_{p+p'}))|=k-a-b\geq 1$, then
$d_{T'}(S)=d_T(S).$

\item If $|S\cap V(T_{\geq j+1})|=b\geq 1$  and $|S \cap V'(P(u_{p'},u_{p+p'}))|=k-a-b\geq 1$, then
$d_{T'}(S)=d_T(S)+(p'-p).$
\end{itemize}

By summing the above six subcases (but with fixed values of $x_1$ and $x_2$), we have the total change in $d(S)$ from $T$ to $T'$ as
\begin{align*}
   &\sum_{a=1}^{k-2}\sum_{b=1}^{k-1-a}\binom{|T_j|-1}{a}\binom{p'-p+1}{k-a-b}\left(x_1\binom{|T_{\geq j+1}|}{b}-\left(p'-p-x_2\right)\binom{|T_{\leq j-1}|}{b}\right)\\
   &\quad\quad\quad +\sum_{a=1}^{k-2}\sum_{b=1}^{k-1-a}(p'-p)\binom{|T_j|-1}{a}\binom{p-1}{k-a-b}\left(\binom{|T_{\geq j+1}|}{b}-\binom{|T_{\leq j-1}|}{b}\right)\\
 =&\sum_{a=1}^{k-2}\sum_{b=1}^{k-1-a}\binom{|T_j|-1}{a}\binom{p'-p+1}{k-a-b}\left(\left(x_1+x_2-p'+p\right)\binom{|T_{\leq j-1}|}{b}\right)\\
 &\quad\quad\quad +\sum_{a=1}^{k-2}\sum_{b=1}^{k-1-a}x_1\binom{|T_j|-1}{a}\binom{p'-p+1}{k-a-b}\left(\binom{|T_{\geq j+1}|}{b}-\binom{|T_{\leq j-1}|}{b}\right) \\
 &\quad\quad\quad +\sum_{a=1}^{k-2}\sum_{b=1}^{k-1-a}(p'-p)\binom{|T_j|-1}{a}\binom{p-1}{k-a-b}\left(\binom{|T_{\geq j+1}|}{b}-\binom{|T_{\leq j-1}|}{b}\right)\\
 =: & \Delta_{x_1, x_2} + \Delta_{x_1} +\Delta.
\end{align*}

Through similar arguments as before, it is easy to see that $\Delta_{x_1} \geq 0, \Delta> 0$ and
$$ \sum_{x_1, x_2} \Delta_{x_1, x_2} = 0 $$
since $\Delta_{x_1, x_2} + \Delta_{p'-p-x_2, p'-p-x_1} = 0$.

\item If $|S\cap V'(T_j)|=a\geq 1$, and $S$ contains some other vertices from at least two of
the paths $P(u_0,u_p)$, $P(u_p,u_{p'})$ and $P(u_{p'},u_{p+p'})$:

\begin{itemize}
\item If $|S\cap V(T_{\leq j-1})|=b\geq 1$ and the remaining vertices of $S$ are all in $P(u_0,u_p)$ and $P(u_p,u_{p'})$, let $u_{p+x_2}$ be such a vertex with the largest index
for some $0\leq x_2\leq p'-p$, then
$d_{T'}(S)=d_T(S)+x_2-(p'-p).$

\item If $|S\cap V(T_{\geq j+1})|=b\geq 1$ and the remaining vertices of $S$ are all in $P(u_p,u_{p'})$ and $P(u_{p'},u_{p+p'})$,
let $u_{p+x_1}$ be such a vertex with the smallest index for $0\leq x_1\leq p'-p$, then
$d_{T'}(S)=d_T(S)+x_1.$

\item In any other cases we have $d_{T'}(S)=d_T(S).$

\end{itemize}

Similar arguments as the previous cases show that the total change of value in $d(S)$ (in the above three subcases), from $T$ to $T'$, is non-negative.

\end{itemize}

Thus
$$  SW_k(T') - SW_k(T) > 0, $$
a contradiction. Hence we  must have $|T_{\leq j-1}| \geq |T_{\geq j+1}|$, and consequently
$$ |T_{\leq i-1}| > |T_{\leq j-1}| \geq |T_{\geq j+1}| > |T_{\geq i+1}| $$
for any $i> j$, implying that $r_{i} \geq r_{i+1}$ by the same argument. It follows that $r_j \geq r_{j+1} \geq \cdots$. Similarly, one can show that $r_1 \leq \cdots \leq r_j$.

\item For simplicity we only consider the case when all branch vertices have degree 3. All other cases (with one or two vertices having degree 4) can be argued in exactly the same way.

 Let $S_i$ denote the pendant segment at $v_i$ ($1\leq i \leq k' = k-1$), let $s_i$ denote the length of $S_i$, and let $\mu$ be the minimum length of all pendant segments.

Similar to before we may assume that there exists a smallest index $j'$ such that $s_{j'} = \mu < s_{j'+1}$,
then we have $|T_{\leq j'} - S_{j'}| \geq |T_{\geq j'+1} - S_{j'+1}|$, or interchanging $S_{j'}$ and $S_{j'+1}$ will increase the Steiner Wiener index by
Lemma~\ref{lemma3.1}.
Thus
$$|T_{\leq i} - S_i| \geq |T_{\leq j'}| > |T_{\leq j'} - S_{j'}| \geq |T_{\geq j'+1} - S_{j'+1}| \geq |T_{\geq i+1}| > |T_{\geq i+1} - S_{i+1}|$$
for any $i > j'$, which implies that $s_{i+1} \geq s_i$ by the same argument.  It follows that $s_{j'} \leq s_{j'+1} \leq \cdots$. One can show that $s_1 \geq \cdots \geq s_{j'}$ in exactly the same way.
\end{enumerate}

\end{proof}

\section{Extremal trees with a given number of segments}
\label{sec:num}

It is also interesting to examine the extremal problems among trees with given order and number of segments. For those that minimizes the original Wiener index it was shown in \cite{lin2015} to be the so-called {\it balanced starlike trees} (starlike trees whose segment lengths differ by no more than 1).

\begin{theo}\label{theo:theo1}
Given the number of segments and the number of vertices, the balanced starlike tree minimizes the Steiner $k$-Wiener index for any $k$.
\end{theo}

\begin{proof}
Let such an optimal tree be $T$ with $n$ vertices and $m\geq 3$ segments, with segment sequence $(l_1, l_2, \cdots, l_m)$. From Theorem~\ref{theo:theo} we know $T$ is a starlike tree. It remains to show that $|l_i - l_j| \leq 1$ for any $1 \leq i \leq j \leq m$.

Otherwise, suppose,
without loss of generality, that $l_2< l_1-1$. Let $v_0$ be the unique branch vertex in $T$ with $ v_0u_1\cdots u_{l_1}$ and $v_0v_1 \cdots v_{l_2}$ being two pendent segments of length $l_1$ and $l_2$, respectively. We denote by
$T_0$ the component containing $v_0$ in $T- v_0u_1 - v_0v_1$. Consider, now, the tree $T'$ obtained from $T$
by ``sliding'' $T_0$ from $v_0$ to $u_1$ (Figure~\ref{fig:e1}).

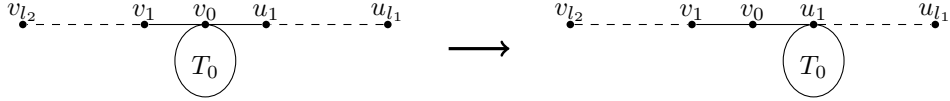
\begin{figure}[htbp]
\centering
    \begin{tikzpicture}[scale=.8]
        \node[fill=black,circle,inner sep=1pt] (t1) at (0,0) {};
        \node[fill=black,circle,inner sep=1pt] (t2) at (1,0) {};
		\node[fill=black,circle,inner sep=1pt] (t3) at (2,0) {};
        \node[fill=black,circle,inner sep=1pt] (t4) at (4,0) {};
        \node[fill=black,circle,inner sep=1pt] (t5) at (-2,0) {};

        \draw (t1)--(t3);
        \draw [dashed] (t1)--(t5) (t4)--(t3);
       \draw (1,-.6)  ellipse [x radius=.5cm, y radius=0.6cm];
        \node at (4,.2) {$u_{l_1}$};
        \node at (2,.2) {$u_{1}$};
        \node at (1,.2) {$v_{0}$};
        \node at (0,.2) {$v_{1}$};
        \node at (-2,.2) {$v_{l_2}$};
                \node at (1,-0.7) {$T_{0}$};

         \draw [->,very thick] (5,-0.4)--(6,-.4);

     \node[fill=black,circle,inner sep=1pt] (t1) at (7,0) {};
        \node[fill=black,circle,inner sep=1pt] (t2) at (9,0) {};
		\node[fill=black,circle,inner sep=1pt] (t3) at (10,0) {};
        \node[fill=black,circle,inner sep=1pt] (t4) at (11,0) {};
        \node[fill=black,circle,inner sep=1pt] (t5) at (13,0) {};

        \draw (t2)--(t4);
        \draw [dashed] (t1)--(t2) (t4)--(t5);

 \draw (11,-.6)  ellipse [x radius=.5cm, y radius=0.6cm];
        \node at (11,.2) {$u_1$};
        \node at (10,.2) {$v_{0}$};
        \node at (9,.2) {$v_{1}$};
        \node at (7,.2) {$v_{l_2}$};
        \node at (13,.2) {$u_{l_1}$};
                \node at (11,-0.7) {$T_0$};

        \end{tikzpicture}
\caption{The tree $T$ (on the left) and the tree $T'$ (on the right).}\label{fig:e1}
\end{figure}

Similar to before we may discuss the change in $d(S)$ from $T$ to $T'$ depending on the choices of $S$. Here we use $V'(T_0)$ to denote $V(T_0)$ without the vertex on the path from $v_{l_2}$ to $u_{l_1}$.

\begin{itemize}
\item If $|S\cap V'(T_0)|=a\geq 1$, $|S\cap \{u_1, u_2, \cdots, u_{l_1}\}|=0$, and $|S\cap \{v_0, v_1, \cdots, v_{l_2}\}|=k-a\geq 1$, then
$d_{T'}(S)=d_{T}(S)+1.$
\item If $|S\cap V'(T_0)|=a\geq 1$, $|S\cap \{u_1, u_2, \cdots, u_{l_1}\}|=k-a\geq 1$, $|S\cap \{v_0, v_1, \cdots, v_{l_2}\}|=0$, then
$d_{T'}(S)=d_{T}(S)-1.$
\item In other cases, $d_{T'}(S)=d_{T}(S).$
\end{itemize}

Summing over all cases, we have
\begin{align*}
 SW_k(T) - SW_k(T') & = \sum_{S\subseteq V(T),|S|=k}d_{T}(S)-d_{T'}(S)\\
 & = \sum_{a=1}^{k-1}\binom{|T_0|-1}{a}\cdot \left(\binom{l_1}{k-a}-\binom{l_2+1}{k-a}\right)\\
 & > 0,
\end{align*}
a contradiction.
\end{proof}

Next we consider the trees, with given order and number of segments, that maximize the Steiner $k$-Wiener index. With Theorems~\ref{theo:theo'} and \ref{thm:further} we only need to consider quasi-caterpillars with vertex degrees $\leq 4$. First we describe four special trees, each of which is a caterpillar of order $n$ with $m$ segments, with internal vertices $v_1, \ldots , v_{t-1}$ on a path $P(v_0, v_t)= v_0v_1 \ldots v_{t-1} v_t$:

\begin{itemize}
\item The tree $T_i$ ($t=\frac{2n-m-1}{2}$): $d(v_1)=d(v_{t-1})=4$,  $d(v_2)=d(v_3)=\cdots=d(\lfloor\frac{m-7}{4}\rfloor+1)=3$, $d(v_{t-2})=d(v_{t-3})=\cdots=d(v_{t-\lceil\frac{m-7}{4}\rceil-1})=3$, and all other internal vertices have degree $2$.

\item The tree $T_{ii}$ ($t=\frac{2n-m+1}{2}$):  $d(v_1)=d(v_2)=\cdots=d(v_{\lfloor\frac{m-1}{4}\rfloor})=3$, $d(v_{t-1})=d(v_{t-2})=\cdots=d(v_{t-\lceil\frac{m-1}{4}\rceil})=3,$  and all other internal vertices have degree $2$.

\item The tree $T_{iii}$ ($t=\frac{2n-m}{2}, m\equiv 0  \mod 4$):
$d(v_1)=4$, $d(v_2)=d(v_3)=\cdots=d(v_{\frac{m}{4}-1})=3$,
$d(v_{t-1})=d(v_{t-2})=\cdots=d(v_{t-\frac{m}{4}})=3,$ and all other internal vertices have degree $2$.

\item The tree $T_{iv}$ ($t=\frac{2n-m}{2}, m\equiv 2 \mod 4$):
$d(v_1)=4$, $d(v_2)=d(v_3)=\cdots=d(v_{\lfloor\frac{m-4}{4}\rfloor+1})=3$,
$d(v_{t-1})=d(v_{t-2})=\cdots=d(v_{t-\lceil\frac{m-4}{4}\rceil})=3,$ and all other internal vertices have degree $2$.

\end{itemize}

\begin{theo}\label{theo:theo2}
For any $k$, among trees of order $n$ with $m$ segments, the Steiner $k$-Wiener index is maximized by one of the caterpillars $T_i$, $T_{ii}$, $T_{iii}$, $T_{iv}$.
\end{theo}

\begin{proof}
Let $T$ be such an optimal tree of order $n$ with $m$ segments. As already mentioned, from Theorems~\ref{theo:theo'} and \ref{thm:further} we may assume $T$ to be a caterpillar with vertex degree no more than 4 and at most two vertices of degree 4 (see Theorem~\ref{thm:further}). We only consider the case with all internal vertex of degree 3 here. All other cases are similar.

Let the backbone be the longest path $P(u_0, u_{a+1})$, with leaves $u_0$ and $u_{a+1}$, and branch vertices $u_1, \cdots, u_{a}$. Note that $m=2a+1$ from our assumption on the number of segments and vertex degrees.

Let $l_1$ and $l_2$ be the lengths of $P(u_0, u_1)$ and the other pendent segment ending at $u_1$. First we show that $l_2=1$.

Otherwise, we have $l_1 \geq l_2 >1$. If we replace the two segments by segments of length $1$ and $l_1+l_2-1$, the Steiner Wiener index
will increase (by arguments similar to that of Theorem~\ref{theo:theo1}), a contradiction.

Thus the pendent segments at $v_1$, and for the same reasoning, $v_{k-1}$, have to of length $1$. Then, by statement (3) of Theorem~\ref{thm:further},
all pendent segments must have length $1$. In other words, $T$ is a caterpillar.

From the study of trees with given degree sequence that maximize the Steiner Wiener index \cite{jie2018+}, we know
that the degrees of the internal vertices along the backbone have to be decreasing first, then increasing, i.e.,
the sequence of degrees has to be the form $3, 3,\cdots, 3, 2, 2,\cdots, 2,3, \cdots, 3$.

Lastly we show that the vertices of degree $3$ are ``evenly distributed'' on the two sides of the backbone. To show this we relabel the vertices (including both branch vertices and other vertices) on the backbone first: $w_0=u_0, w_1, \cdots, w_{a+c+1}$. Suppose there is a pendant edge at each of $w_1, \cdots, w_x$ and $w_{x+c+1}, \cdots, w_{a+c}$ with $c=n-m-1$
being the number of vertices of degree $2$ in $T$.

Suppose, for contradiction, that $c>0$, and $x> a-x+1$.
Let $T'$ be the tree obtained from $T$ by moving one pedant edge from $w_x$ to $w_{x+c}$ (Figure~\ref{fig:ex5}).

\begin{figure}[htbp]
\centering
    \begin{tikzpicture}[scale=.8]
        \node[fill=black,circle,inner sep=1pt] (t1) at (-4,0) {};
        \node[fill=black,circle,inner sep=1pt] (t2) at (-3,0) {};
		\node[fill=black,circle,inner sep=1pt] (t3) at (-1,0) {};
        \node[fill=black,circle,inner sep=1pt] (t4) at (0,0) {};
        \node[fill=black,circle,inner sep=1pt] (t5) at (1,0) {};
        \node[fill=black,circle,inner sep=1pt] (t6) at (3,0) {};
        \node[fill=black,circle,inner sep=1pt] (t7) at (4,0) {};
         \node[fill=black,circle,inner sep=1pt] (t8) at (6,0) {};
         \node[fill=black,circle,inner sep=1pt] (t9) at (7,0) {};
        \node[fill=black,circle,inner sep=1pt] (a1) at (-3,-1) {};
        \node[fill=black,circle,inner sep=1pt] (a2) at (-1,-1) {};
         \node[fill=black,circle,inner sep=1pt] (a3) at (0,-1) {};
         \node[fill=black,circle,inner sep=1pt] (a4) at (4,-1) {};
          \node[fill=black,circle,inner sep=1pt] (a5) at (6,-1) {};
       \draw (t1)--(t2) (t3)--(t5) (t7)--(t6) (t9)--(t8);
       \draw [dashed] (t2)--(t3) (t5)--(t6) (t7)--(t8);
        \draw (t2)--(a1) (t3)--(a2) (t4)--(a3) ;
        \draw  (t7)--(a4);
        \draw (t8)--(a5);

        \node at (-4,0.2) {$w_{0}$};
        \node at (-3,0.2) {$w_{1}$};
        \node at (-1,.2) {$w_{x-1}$};
        \node at (0,0.2) {$w_{x}$};
        \node at (1,0.2) {$w_{x+1}$};
        \node at (3,0.2) {$w_{x+c}$};
        \node at (4.4,0.2) {$w_{x+c+1}$};
        \node at (6,0.2) {$w_{a+c}$};
         \node at (7.4,0.2) {$w_{a+c+1}$};

        \end{tikzpicture}
\caption{An extremal caterpillar.}\label{fig:ex5}
\end{figure}
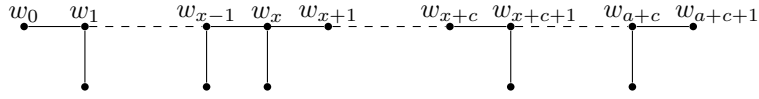

Note that the transformation from $T$ to $T'$ can be simply considered as ``switching'' the pendant edge at $w_x$ with the single vertex $w_{x+c}$, direct application of Lemma~\ref{lemma3.1} shows
$SW_k(T')> SW_k(T)$, a contradiction.

Summarizing the above, the optimal tree here is exactly $T_{ii}$ as described above.
\end{proof}

\section{Concluding remarks}
\label{sec:con}
We considered the extremal problems with respect to the Steiner Wiener index among trees with a given segment sequence. For the minimizing case, the extremal tree is shown to be a starlike tree and coincides with that for the original Wiener index. For the maximizing case, the extremal tree is shown to be a quasi-caterpillar with some additional properties. This is also similar to what is known about the original Wiener index. It would be interesting, however, to see if the extremal quasi-caterpillar differs for Steiner $k$-Wiener index for different values of $k$; and if so, how big is the difference.

Extremal problems among trees of given order with given number of segments are also considered and the extremal trees are characterized. In the maximizing case, it is shown that the extremal tree has to be one of several caterpillars. Further investigation of more exact characterization would be interesting.

\end {document}